\documentclass[leqno,english]{amsproc}
\usepackage{amsfonts,amssymb,amsmath,amsgen,amsthm}
\usepackage{graphicx}
\usepackage{tikz,tkz-tab}
\usepackage{hyperref}
\usepackage{color}
\newcommand{\msc}[2][2000]{%
  \let\@oldtitle\@title%
  \gdef\@title{\@oldtitle\footnotetext{#1 \emph{Mathematics subject
        classification.} #2}}%
}

\theoremstyle{plain}
\newtheorem{theo}{Theorem} [section]
\newtheorem{defi}[theo]{Definition}

\newtheorem{lem}[theo]{Lemma}
\newtheorem{cor}[theo]{Corollary}
\newtheorem{prop}[theo]{Proposition}

\theoremstyle{remark}
\newtheorem{rem}[theo]{Remark}
\newtheorem{nota}[theo]{Notation}
\newtheorem{example}[theo]{Example}

\def\R{{\mathbb R}}

\def\({\left(}
\def\){\right)}
\def\<{\left\langle}
\def\>{\right\rangle}

\def\Eq#1#2{\mathop{\sim}\limits_{#1\rightarrow#2}}
\def\Tend#1#2{\mathop{\longrightarrow}\limits_{#1\rightarrow#2}}

\def\d{{\partial}}
\def\ep{\varepsilon}

\numberwithin{equation}{section}

\begin{document}

\title[Superexponential regimes in the logarithmic heat equation]{Superexponential growth or decay in the heat equation with a
  logarithmic nonlinearity} 

\author[M. Alfaro]{Matthieu Alfaro}
\address{CNRS\\Institut Montpelli\'erain Alexander Grothendieck\\
  Univ. Montpellier 
\\CC51\\Place E. Bataillon\\34095 Montpellier\\ France}
\email{Matthieu.Alfaro@umontpellier.fr}

\author[R. Carles]{R\'emi Carles}
\address{CNRS\\Institut Montpelli\'erain Alexander Grothendieck\\
  Univ. Montpellier 
\\CC51\\Place E. Bataillon\\34095 Montpellier\\ France}
\email{Remi.Carles@math.cnrs.fr}

\begin{abstract}
We consider the heat equation with a logarithmic nonlinearity, on the
real line. For a suitable sign in front of the nonlinearity, we
establish the existence and uniqueness of solutions of the Cauchy
problem, for a well-adapted class of initial data. Explicit
computations in the case of Gaussian data lead to various scenarii
which are richer than the mere comparison with the ODE mechanism,
involving (like in the ODE case) double exponential growth or decay
for large time. Finally, 
we prove that such phenomena remain, in the case of compactly
supported initial data.   
\end{abstract}

\maketitle

 \tableofcontents

\section{Introduction}\label{s:intro}

In this work we consider the nonnegative solutions $u(t,x)$ of the heat equation with a logarithmic nonlinearity, namely 
\begin{equation}
\label{eq}
\partial _t u=\partial _{xx} u +\lambda\, u \ln \(u^2\), \quad t>0,\; x\in \R,
\end{equation}
where $\lambda >0$ is a given parameter. Our primary goal is to
investigate the large time behavior of the solutions, since this study
reveals mechanisms which seem interesting to us. Equation \eqref{eq} shares some similarities with  {\it bistable} equations modelling an {\it Allee effect} in population dynamics, but does not seem to correspond clearly to any model proposed in e.g. biology or chemistry. On the other hand,  \eqref{eq} has challenging aspects from the mathematical point of view.  Our results
may be extended to the multidimensional case, leading to a more
technical setting. We have chosen to stick to the one-dimensional
case to simplify the presentation, thus highlighting the main mechanisms.

Associated with \eqref{eq} is the following energy
\begin{equation}
\label{def-energie}
\mathcal E [u](t):=\frac 12 \int _\R (\partial _x u)^{2}(t,x)dx+\int _\R \frac \lambda 2  u^{2}\(1-\ln \(u^2\) \)(t,x)dx.
\end{equation}
Formally, solutions to \eqref{eq} satisfy 
$$
\frac{d\mathcal E [u]}{dt}=-\int _\R (\partial _t u)^{2}(t,x)dx \leq 0.
$$

Many features make \eqref{eq} interesting from a mathematical
point of view. First, the nonlinearity is not Lipschitzean, which
causes difficulties already at the level of the local Cauchy
problem. Also, the second term in the energy \eqref{def-energie} has
no definite sign, which makes  a priori estimates a delicate issue. Next, \eqref{eq} supports the 
existence of Gaussian solutions. Last, the Cauchy problem may exhibit superexponential growth or decay.

\subsection{The Cauchy problem}

Such a logarithmic nonlinearity has been introduced in Physics in the
context of wave mechanics and optics \cite{BiMy75,BiMy76}. From a mathematical
point of view, the Cauchy problem for logarithmic Schr\"odinger
equations and logarithmic wave equations have been studied in
\cite{CaHa80,CaGa-p}: in the case of the logarithmic Schr\"odinger
equation, it is shown that
a unique, global weak solution can be constructed in a subset of $H^1$
(in any space dimension), whichever the sign of $\lambda$. For the
three dimensional wave equation and a suitable sign for $\lambda$, a
similar result is available. Due to the lack of regularity of the
nonlinearity, solutions are constructed by compactness methods, and
uniqueness is a rather unexpected property: in the case of
Schr\"odinger equation, it is a consequence of an elegant estimate in complex
analysis noticed in \cite{CaHa80}, while for the three dimensional
wave equation, it follows from fine properties of the wave equation
and a general result concerning the trace (see \cite{CaHa80} or
\cite{Ha81}).  
\smallbreak

In the context of the heat equation like \eqref{eq}, the presence of a
logarithmic nonlinearity has been considered in \cite{ChLuLi15}, in
the case of a bounded domain $\Omega$, with Dirichlet boundary conditions. They
construct global solutions in $H^1_0(\Omega)$, and exhibit some
classes of solution growing or decaying  (at least) exponentially in
time, thanks to variational arguments (potential well method). On the other hand, it seems very
delicate, if possible, to 
construct a solution to \eqref{eq} by compactness methods on the whole line
$\R$. Also, uniqueness is missing in the Cauchy theory developed in
\cite{ChLuLi15}. We will see that this issue can be overcome by
changing functional spaces in which the Cauchy problem is studied. 

\begin{defi}[Notion of solution]\label{def:sol-log}
  Let $u_0$ be  continuous and bounded, with continuous and bounded
  derivative, and bounded and piecewise continuous second
  derivative. A (global) solution to \eqref{eq} starting from $u_0$ is a
  function $u:[0,\infty)\times \R\to \R$ which is continuous and bounded on
  $[0,T]\times \R$, for which $u_t$, $u_x$ and $u_{xx}$ exist and
  are continuous on $(0,T]\times \R$, such that $u(t,x)$ solves \eqref{eq} on $(0,T]\times \R$ (for any $T>0$), and $u_{\mid t=0}=u_0$.  In addition, we require that
  $u(t,x)$ is uniformly bounded as $|x|\to \infty$ for $t\in [0,T]$. 
\end{defi}

\begin{prop}[Global well-posedness for \eqref{eq}]\label{prop:cauchy}
Let  $u_0\ge 0$ be as in Definition ~\ref{def:sol-log}. Then \eqref{eq} has a unique solution $u$ starting from $u_0$, in
the sense of Definition~\ref{def:sol-log}. 
\end{prop}






\subsection{Superexponential growth vs. decay} 

In
\cite{JiYiCa16}, the presence of the logarithmic nonlinearity is
motivated as some limiting case for a nonlinearity of the form
$\lambda u^{1+\ep}$ in the limit $\ep\to 0$ (in the growth regime $u\to \infty$), still in a bounded domain
$\Omega\subset \R^N$ with Dirichlet boundary conditions. The authors show in
particular that time periodic solutions are highly unstable, in the
sense that a small perturbation of the initial data can lead to double
exponential growth or double exponential decay in time, see \cite[Theorem 1.1]{JiYiCa16}. Some of our results are qualitatively similar (superexponential
growth or decay), with the Gaussian steady state \eqref{def-varphi}
acting as a separation comparable to the time periodic solutions in
\cite{JiYiCa16}. Nevertheless, let us mention two main
differences. First, our results are valid on the whole line $\R$, and
not (only) on a bounded domain. On the other hand, we provide in Section \ref{s:data} initial data leading to superexponential growth or decay but that cannot be handled by \cite[Theorem 1.1]{JiYiCa16}. Roughly speaking, as can be seen from the proof, initial data of \cite[Theorem 1.1]{JiYiCa16} are multiples of the separating time periodic solution --- which is comparable to the present Remark \ref{rem:dichotomy}--- whereas initial data in Section \ref{s:data} are allowed to \lq\lq cross'' the separating Gaussian steady state \eqref{def-varphi} (see also Corollary \ref{cor:comp-1}).

\smallbreak

Let us recall that, in his seminal work \cite{Fuj-66}, Fujita considered solutions
$u(t,x)$ to the nonlinear heat equation 
\begin{equation}
\label{eq-fujita}
\partial _t u=\Delta u+u^{1+p},\quad t>0,\; x\in \R^{N},
\end{equation}
supplemented with a nonnegative and nontrivial initial data. For $p>0$
solutions of the underlying ordinary differential equation (ODE)
problem --- namely $\frac{dn}{dt}=n^{1+p}$, $n(0)=n_0>0$ --- blow up in
finite time. The dynamics of the partial differential equation
\eqref{eq-fujita} is more complex 
and rich. Precisely, there is a critical exponent $p_F:=\frac 2 N$,
referred to as the {\it Fujita exponent}, such that: If $0<p\leq p_F$
then any solution blows up in finite time, like those of the ODE. On the
other hand, if $p>p_F$ there is a balance between diffusion and
reaction. Solutions with large initial data blow up in finite time
whereas solutions with small initial data are global in time and go
extinct as $t\to\infty$. Those facts are proved in \cite{Fuj-66},
except the critical case $p=p_F$ which is studied in \cite{Hay-73}
when $N=1,2$, in \cite{Kob-Sir-Tun-77} when $N\geq 3$, and in
\cite{Wei-81} via a direct and simpler approach.  

Concerning equation \eqref{eq}, the underlying ODE problem
\begin{equation}
\label{ode}
\frac{dn}{dt}=2\lambda n \ln n, \quad n(0)=n_0>0,
\end{equation}
is globally solved as
\begin{equation}
\label{sol-ode}
n(t)=e^{(\ln n_0)e^{2\lambda t}}.
\end{equation}
As $t\to \infty$, $n(t)\to 0$ if $0<n_0<1$ (extinction) whereas
$n(t)\to \infty$ if $n_0>1$ (blow up in infinite time). Hence the
dynamics of the ODE \eqref{ode} already shares some similarities  with
the Fujita supercritical regime $p>p_F$ for the PDE
\eqref{eq-fujita}. The dynamics of the PDE \eqref{eq} is much richer
than the mechanism of \eqref{ode}, and our main goal is to understand its long time behavior for initial data $u_0$ ``crossing'' the equilibrium 1.

Notice also that the composition of exponential functions in
\eqref{sol-ode} is a strong indication that possible extinction or
growth phenomena are strong, and can thus hardly be
captured numerically. In practice, the superexponential growth may
appear like a blow-up phenomenon, while the superexponential decay may
be understood like a finite time extinction.  

 When possible, a second goal is to estimate
these rates of convergence. 

\medskip 

To give a flavor of the results established in the sequel, recall that the authors in \cite{ChLuLi15}  consider
\eqref{eq} (possibly in multidimension) on a bounded domain, with
Dirichlet boundary conditions. By variational arguments, they exhibit
classes of initial data whose evolution under \eqref{eq} leads to  (at
least) exponential decay in $L^2$, and another class of initial data
whose evolution leads to unboundedness of the $L^2$ norm in large
time. As a consequence of our analysis on the whole line $\R$, we actually provide more precise information on those phenomena for the equation on a bounded domain, say $(\alpha,\beta)$. 

\begin{prop}[Growth/decay rates in  a bounded domain]\label{prop:Omega} 
  Let $\alpha<\beta$ and $\Omega=(\alpha,\beta)$. Consider the mixed
  problem
\begin{equation}\label{eq:Omega}
\left\{
  \begin{aligned}
    \d_t u &= \d_{xx}u + u\ln (u^2),\quad t>0,\ x\in \Omega,\\
u_{\mid \d \Omega}&=0,\quad t>0,\\
u_{\mid t=0}&=u_0.
  \end{aligned}
\right.
\end{equation}
There are nonnegative initial data $u_0\in C^1_c(\Omega)$ such that
\eqref{eq:Omega} has a unique solution, whose $L^2$ and $L^\infty$ norms  decay at 
least like a double exponential in time,
\begin{equation*}
  \exists C,\eta>0,\quad  \|u(t)\|_{L^2(\Omega)}\le
  |\Omega|^{1/2}\|u(t)\|_{L^\infty(\Omega)}\le C e^{-\eta e^{2t}}. 
\end{equation*}
There are nonnegative initial data $u_0\in C^1_c(\Omega)$ such that
\eqref{eq:Omega} has a unique solution, whose $L^2$ and $L^\infty$ norms grow at 
least like a double exponential in time, 
\begin{equation*}
  \exists C,\eta>0,\quad  |\Omega|^{1/2}\|u(t)\|_{L^\infty(\Omega)}
  \ge  \|u(t)\|_{L^2(\Omega)}\ge C e^{\eta e^{2t}}.
\end{equation*}
\end{prop}

\subsection{Changing the sign of the nonlinearity} 
Let us observe that, for $\lambda >0$, the problem 
\begin{equation}
\label{eq-moins}
\partial _t u=\partial _{xx} u -2\lambda\, u \ln u, \quad t>0,\; x\in \R,
\end{equation}
is of different nature. Indeed for the underlying ODE,
$$
n'(t)=-2\lambda n(t)\ln n(t),
$$
the equilibrium 0 is (very) unstable,
while 1 is stable. Hence, by the comparison principle, solutions are
{\it a priori} bounded between 0 and $\max (1,\Vert u_0\Vert
_{L^\infty})$. Moreover, by comparison with Fisher-KPP equations, much
can be said on the long time behavior of the Cauchy problem. For
instance, consider a nontrivial compactly supported initial data
$0\leq u_0\leq 1$. For any $r>0$, we can construct  
$$
g_r:[0,1]\to \R \text{ concave with } g_r>0 \text{ on } (0,1), \; g_r(0)=g_r(1), \; r=g'_r(0)>0>g'_r(1),
$$
which is referred as to a Fisher-KPP nonlinearity, and such that
$g_r(u)\leq -2\lambda u \ln u$. By the comparison principle, we deduce
that $u_r(t,x)\leq u(t,x)\leq 1$, where $u_r$ is the solution of 
$$
\partial _t u_r=\partial _{xx} u_r +g_r(u_r), 
$$
starting from $u_0$. But it is known \cite{Aro-Wei-78} that the
spreading speed of this Fisher-KPP equation, with compactly supported
data, is $c_r^*:=2\sqrt{g_r'(0)}=2\sqrt r$, meaning that 
$$
\text{ if } c>c_r^* \text{ then } u_r(t,x)\to 0 \text{ uniformly in } \{\vert x\vert \geq ct\} \text{ as } t\to \infty,
$$ 
$$
\text{ if } c<c_r^* \text{ then } u_r(t,x)\to 1 \text{ uniformly in } \{\vert x\vert \leq ct\} \text{ as } t\to \infty.
$$ 
Since this is true for any $r>0$ we get that
$$
\text{ for any } c>0,\, u(t,x)\to 1 \text{ uniformly in } \{\vert x\vert \leq ct\} \text{ as } t\to \infty,
$$
that is convergence to 1 with a superlinear speed.

\medskip

The organization of the paper is as follows. In Section
\ref{s:steady}, we enquire on steady states, proving existence of a
unique (Gaussian) nontrivial one.  The well-posedness of the Cauchy
problem is established in Section~\ref{s:cauchy}. The long time behavior
(superexponential growth, decay or convergence to the steady state) is
studied in Section~\ref{s:gauss} (Gaussian initial data and
consequences), and Section~\ref{s:data} (more general data and
consequences). 
 
\section{Steady states}\label{s:steady}

It is readily checked that the only constant steady states of
\eqref{eq} are $u\equiv 0$ and $u\equiv 1$. 

\begin{prop}[Steady state]\label{prop:steady} There is a unique (up to
  translation) nonnegative nontrivial steady state $\varphi$
  solving \eqref{eq} and satisfying $\varphi(\pm \infty)=0$. It is the
  Gaussian given by 
\begin{equation}
\label{def-varphi}
\varphi(x)=e^{\frac 1 2}e^{-\frac \lambda 2 x^2}.
\end{equation}
\end{prop}

\begin{proof} Let $u=u(x)\geq 0$ be a nontrivial solution to \eqref{eq}, that is 
\begin{equation}
\label{eq-steady}
u''(x)+2\lambda u(x)\ln u(x)=0, \quad \forall x \in \R,
\end{equation}
 with $u(\pm \infty)=0$. If $u(x_0)=0$ for some $x_0\in\R$ then $u\equiv 0$ from the strong maximum principle. Hence $u>0$. Next, we multiply the equation by $u'$, integrate and infer that there is $C\in \R$ such that
$$
\({u'}\)^2(x)+ \lambda u^2(x)(2\ln u(x)-1)=C, \quad \forall x\in \R.
$$
{}From the above identity and since $u(\pm\infty)=0$,  we deduce that
$u'(\pm \infty)$ must exist in $\R$ and, thus, be equal to 0
(otherwise we cannot have $u(\pm \infty)=0$). Hence $C=0$ and 
\begin{equation}
\label{eq-steady-int}
\({u'}\)^2(x)=\lambda u^2(x)(1-2\ln u(x)), \quad \forall x\in \R.
\end{equation}
If $x\mapsto 1-2\ln u(x)$ never vanishes then this identity implies
that $u'$ has a constant sign, which contradicts $u(\pm
\infty)=0$. Hence, there exists $x_0\in \R$ such that $2\ln
u(x_0)=1$, and thus $u'(x_0)=0$ (from \eqref{eq-steady-int}),
$u''(x_0)<0$ (from \eqref{eq-steady}). In the sequel, we work
on $[x_0,+\infty)$, the arguments being similar on
$(-\infty,x_0]$. 
\smallbreak

Assume that there is $x_1> x_0$ such that
$u'(x_1)=0$. From \eqref{eq-steady-int}, $2\ln u(x_1)=1$, and there
must be a point $x^{*}\in(x_0,x_1)$ where $u$ reaches a minimum
strictly smaller than $e^{\frac 1 2}$, which contradicts
\eqref{eq-steady-int}. Hence $u'<0$ on $(x_0,+\infty)$. It therefore
follows from \eqref{eq-steady-int} that $-u'(x)=\sqrt \lambda
u(x)\sqrt{1-2\ln u(x)}$ for $x\geq x_0$. Separating variables we get
$$
-\sqrt \lambda (x-x_0)=\int _{u(x_0)}^{u(x)}\frac{du}{u\sqrt{1-2\ln u}}=-\sqrt{1-2\ln u(x)},
$$
since $u(x_0)=e^{\frac 12}$. We end up with $u(x)=e^{\frac
  12}e^{-\frac \lambda 2 (x-x_0)^{2}}$, which completes the proof.
\end{proof}

\section{Cauchy problem}
\label{s:cauchy}

As emphasized in the introduction, the Cauchy problem associated to
\eqref{eq} is not trivial, for two reasons:
\begin{itemize}
\item Local well-posedness: the nonlinearity is not Lipschitzean.
\item Global well-posedness: the potential energy in
  \eqref{def-energie} has no definite sign.  
\end{itemize}
The first aspect implies that constructing a solution certainly requires
compactness arguments, and uniqueness is not granted. The second
aspect shows that to have a solution defined for all $t\ge 0$, it may
be helpful that the first step yields this property ``for free''. This
is the strategy adopted in \cite{ChLuLi15}, where, on a bounded domain
$\Omega$, with Dirichlet boundary conditions, the authors construct a
solution in $H^1_0(\Omega)$ by Galerkin approximation. However,
uniqueness is not established in this case.
\medskip

In this section, we prove Proposition~\ref{prop:cauchy},
by showing that it fits perfectly into the framework of the PhD thesis
of J.~C.~Meyer \cite{Me13}. 
Instead of working in  spaces where the energy  \eqref{def-energie}
is well defined, we adopt the approach of \cite{Me13}, see also \cite{MeNe15}. Consider more generally
the Cauchy problem 
\begin{equation}
  \label{eq:meyer}
  u_t=u_{xx}+f(u),\quad 0<t\leq T,\; x\in \R,\quad u_{\mid t=0}=u_0,
\end{equation}
so we can emphasize which are the suitable assumptions of the
nonlinearity $f$ described in \cite{Me13}. Notice that, in \cite{Me13,MeNe15}, the standard examples, motivated by models
  from Chemistry, are of the form  $f(u) = \pm(u^p)^+$, $0<p<1$, and
  $f(u)=(u^p)^+ ((1-u)^q)^+$, $0<p,q<1$. 
  
The generalization of Definition~\ref{def:sol-log}, as introduced in
\cite{Me13}, is the following.

\begin{defi}[Notion of solution]\label{def:sol}
  Let $u_0$ be  continuous and bounded, with continuous and bounded
  derivative, and bounded and piecewise continuous second
  derivative. A solution to \eqref{eq:meyer}  is a
  function $u:[0,T]\times \R\to \R$ which is continuous and bounded on
  $[0,T]\times \R$,  for which $u_t$, $u_x$ and $u_{xx}$ exist and
  are continuous on $(0,T]\times \R$, such that $u(t,x)$ satisfies
  \eqref{eq}. In addition, we require that 
  $u(t,x)$ is uniformly bounded as $|x|\to \infty$ for $t\in [0,T]$. 
\end{defi}
\begin{nota}
  Following \cite{Me13,MeNe15}, we denote by ${\rm BPC}^2(\R)$ the set
  of such  initial data. 
\end{nota}
Two notions are crucial, and correspond exactly to the type of logarithmic
nonlinearity considered in the present paper. 

\begin{defi}[H\"older continuity]
  Let $\alpha\in (0,1)$. A function $f:\R\to \R$ is said to be $\alpha$-\emph{H\"older continuous} if for any closed bounded interval
  $E\subset \R$, there exists a constant $k_E>0$ such that for all
  $x,y\in E$, 
  \begin{equation*}
    |f(x)-f(y)|\le k_E|x-y|^\alpha.
  \end{equation*}
\end{defi}

A notion weaker than the standard notion of Lipschitz continuity turns
out to be rather interesting, as we will see below.

\begin{defi}[Upper Lipschitz continuity]
  A function $f:\R\to \R$ is said to be \emph{upper Lipschitz
    continuous} if $f$ is continuous, and for any closed bounded
  interval $E\subset \R$, there exists a constant $k_E>0$ such that
  for all $x,y\in E$, with $y\ge x$,
  \begin{equation*}
    f(y)-f(x)\le k_E (y-x). 
  \end{equation*}
\end{defi}

Essentially, this property suffices to have a comparison principle,
hence a uniqueness result for \eqref{eq:meyer}. 

\begin{example}
  In the  case of \eqref{eq}, $f(u) = \lambda u \ln(u^2)$. First, $f$ is
  $\alpha$-H\"older continuous for any $\alpha\in (0,1)$. Indeed, for $y>x>0$, we have
  $$
  \vert f(y)-f(x)\vert=2\lambda \left\vert (y-x)\ln y +x\ln\left(1+\frac{y-x}{x}\right)\right\vert\leq 2\lambda \vert y-x\vert (\vert \ln y\vert +1),
  $$
  so that $\frac{\vert f(y)-f(x)\vert}{\vert y-x\vert^{\alpha}}\leq 2\lambda \vert y-x\vert ^{1-\alpha} (\vert \ln y\vert +1)\leq 2\lambda \vert y\vert ^{1-\alpha}(\vert \ln y\vert +1)$, which remains bounded as $y\to 0$. On the other
  hand, even though $f$ is not Lipschitz continuous, we check that for
  $\lambda>0$ (the case of interest in the present paper), $f$ is
  \emph{upper} Lipschitz continuous. Indeed, for $x,y\in E$ bounded,
  with $y>x>0$, Taylor formula yields
  \begin{align*}
    f(y)-f(x)
   & =(y-x) \int_0^1f'\(x+\theta(y-x)\)d\theta\\
&=2\lambda(y-x)\int_0^1(1+\ln)
\(x+\theta(y-x)\)d\theta\\
&\le  2\lambda(y-x) 2\lambda \(1+\sup_{z\in E}\ln z\).
  \end{align*}
The last factor remains bounded as $x\to 0$. It would not be if the infimum was
considered: $f$ is not Lipschitz continuous. 
\end{example}

\begin{defi}[Sub- and super-solutions]
  Let $\underline u, \overline u:[0,T]\times \R$ be continuous on
  $[0,T]\times \R$ and such that $\underline u_t,\underline
  u_x,\underline u_{xx}, \overline u_t,\overline u_x,\overline u_{xx}$
  exist and are continuous on $(0,T]\times\R$. If
  \begin{align*}
    &\underline u_t-\underline u_{xx}-f\(\underline u\)\le 0\le
    \overline u_t-\overline u_{xx}-f\(\overline u\),\quad 0<t\leq T, \; x\in \R,\\
& \underline u(0,x)\le u_0(x)\le \overline u(0,x),\quad \forall x\in \R,
  \end{align*}
and $\underline u, \overline u$ are uniformly bounded as $|x|\to
\infty$ for $t\in [0,T]$, then $\underline u$ is called a
\emph{regular sub-solution}, and $\overline u$ is called a
\emph{regular super-solution} to \eqref{eq:meyer}. 
\end{defi}
\begin{theo}[Comparison; Theorem~7.1 from
  \cite{Me13,MeNe15}]\label{theo:comparaison}
  Let $f$ be upper Lipschitz continuous. If $\underline u$ and
  $\overline u$ and regular sub and super-solutions on $[0,T]\times\R$, respectively, 
  then 
  \begin{equation*}
    \underline u(t,x)\le \overline u(t,x),\quad \forall (t,x)\in
    [0,T]\times\R. 
  \end{equation*}
\end{theo}
\begin{theo}[Uniqueness; Theorem~7.2 from
  \cite{Me13,MeNe15}]\label{theo:unicite}  
  Let $f$ be upper Lipschitz continuous. Then, for any $T>0$, \eqref{eq:meyer} has at
  most one solution in $[0,T]\times\R$. 
\end{theo}
The following statement is a slight modification from the original,
where we add a uniqueness assumption to simplify the presentation. 
\begin{theo}[Existence; Theorem~8.1 and Corollary~8.6 from
  \cite{Me13,MeNe15}]\label{theo:existence}   
  Suppose that $f$ is $\alpha$-H\"older continuous for some $\alpha\in
  (0,1)$, and let $u_0\in {\rm BPC}^2(\R)$. Suppose that uniqueness
  holds for
  \eqref{eq:meyer}. Then \eqref{eq:meyer} has a (unique) solution
  $u:[0,T^*[\times \R$. In addition, either $T^*=\infty$, or
  $\|u(t,\cdot)\|_{L^\infty(\R)}$ is unbounded as $t\to T^*$. 
\end{theo}

\begin{proof}[Proof of Proposition~\ref{prop:cauchy}] As emphasized above, the nonlinearity in \eqref{eq} is
  both $\alpha$-H\"older continuous (for any $\alpha\in (0,1)$) and
  upper Lipschitz continuous. Therefore, Theorem~\ref{theo:unicite}
  implies uniqueness, and Theorem~\ref{theo:existence} yields a
  (unique) maximal solution $u\in C([0,T^*)\times\R)$. 
  
  We conclude
  by showing that the
  solution is global ($T^*=\infty$) 
  thanks to a suitable a priori estimate. The solution of the ODE \eqref{ode} starting from $\Vert u_0\Vert _{L^{\infty}}$, namely
\begin{equation*}
  \overline u(t) = e^{\ln \Vert u_0\Vert
_{L^{\infty}}e^{2\lambda t}}.
\end{equation*}
is a super-solution, while the zero function is obviously a
sub-solution. Theorem~\ref{theo:comparaison} implies that  
\begin{equation*}
 0\le  u(t,x)\le \overline u(t),\quad \forall t\in [0,T^*).
\end{equation*}
We conclude that $T^*=\infty$, and the result follows.
\end{proof}

\begin{cor}[Initial data comparable to 1]\label{cor:comp-1}
  Let $u_0 \in {\rm BPC}^2(\R)$, $u_0\ge 0$, and $\ep\in (0,1)$.  
  \begin{itemize}
  \item If $u_0(x)\ge  1+\ep$ for all $x\in \R$, then $u$ grows at
    least like a double exponential in time: 
    \begin{equation*}
      u(t,x) \ge e^{\ln(1+\ep)e^{2\lambda t}},\quad \forall t\ge 0, \
      \forall x\in \R.
    \end{equation*}
\item If $u_0(x)\le  1-\ep$ for all $x\in \R$, then $u$ decays at
    least like a double exponential in time: 
    \begin{equation*}
      u(t,x) \le e^{\ln(1-\ep)e^{2\lambda t}},\quad \forall t\ge 0,\
      \forall x\in \R.
    \end{equation*}
  \end{itemize}
\end{cor}
\begin{proof}
  This corollary is a straightforward consequence of
  Proposition~\ref{prop:cauchy},  the comparison
  principle (Theorem~\ref{theo:comparaison}), and the ODE case
  \eqref{ode}--\eqref{sol-ode}. 
\end{proof}

\section{Large time behavior: Gaussian data}\label{s:gauss}

Families of Gaussian solutions for nonlinear (and nonlocal) equations can be found in
\cite{B14}, \cite{Alf-Car-14, Alf-Car-17}, in the context of
evolutionary genetics. In the case of a logarithmic nonlinearity, for
the Schr\"odinger equation, it
was observed in \cite{BiMy76} that the flow preserves the Gaussian
structures, and so the resolution of the partial differential equation
boils down to the resolution of ordinary differential equations; see
\cite{CaGa-p} for more details. It is not surprising that the same
holds in the case of \eqref{eq}, and we have indeed:
 
\begin{prop}[Gaussian solutions]\label{prop:gauss} Let $b_0>0$ and $a_0>0$ be given. The solution of \eqref{eq} starting from the Gaussian 
\begin{equation}
\label{gauss-u0}
u_0(x)=b_0 e^{-\frac{a_0}{2}x^2},
\end{equation}
is the Gaussian given by
\begin{equation}
\label{gauss-u}
u(t,x)=b(t)e^{-\frac{a(t)}{2}x^2}:=e^{\psi (t)e^{2\lambda t}}e^{-\frac{a(t)}{2}x^2},
\end{equation}
where
\begin{equation}
\label{psi(t)}
\psi(t)=\ln b_0-\frac{a_0}2\frac{\ln \lambda -\ln (a_0+(\lambda-a_0)e^{-2\lambda t})}{\lambda -a_0},
\end{equation}
with the natural continuation $\psi(t)=\ln b_0 -\frac 1
2(1-e^{-2\lambda t})$ if $a_0=\lambda$, and 
\begin{equation}
\label{a(t)}
a(t)=\lambda \frac{a_0 e^{2\lambda t}}{\lambda -a_0+a_0e^{2\lambda t}}.
\end{equation}
\end{prop}

\begin{proof}
We plug the ansatz \eqref{gauss-u} into equation \eqref{eq}, we identify the $x^0$ and the $x^2$ coefficients to obtain two ordinary differential equations. The first one is the logistic equation
$$
a'(t)=2a(t)(\lambda -a(t)), 
$$
whose solution, starting from $a(0)=a_0$, is given by \eqref{a(t)}. The second one is
$$
b'(t)=2\lambda b(t)\ln b(t)-a(t)b(t).
$$
Denoting $\phi(t):=\ln b(t)$ the above is recast
$$
\phi'(t)=2\lambda \phi(t)-a(t),
$$
whose solution, starting from $\phi(0)=\ln b_0$, is
\begin{equation}
\label{phi(t)}
\phi(t)=\left(\ln b_0 -\int _0^t e^{-2\lambda s}a(s)ds\right)e^{2\lambda t}.
\end{equation}
Next, using \eqref{a(t)} we get
\begin{align*}
\int _0^t e^{-2\lambda s}a(s)ds&= \int _0  ^{t} \frac{\lambda a_0 e^{-2\lambda s}}{a_0+(\lambda -a_0)e^{-2\lambda s}}ds
\\
&=\left\{
  \begin{aligned}
    &\frac{a_0}{-2(\lambda -a_0)}\left(\ln\left(a_0+(\lambda
        -a_0)e^{-2\lambda t}\right)-\ln \lambda\right)\quad \text{ if
    } a_0\neq \lambda \\
& \frac 1 2 (1-e^{-2\lambda t}) \quad \text{ if } a_0=\lambda,
  \end{aligned}
 \right.
\end{align*}
which we plug into \eqref{phi(t)} to get \eqref{psi(t)}.
 \end{proof}
 
 Clearly, the sign of $\psi_\infty:=\lim _{t\to \infty}\psi(t)$
 decides between (superexponential) decay and growth of the Cauchy
 problem starting from a Gaussian data, the critical case $\psi
 _\infty=0$ leading to convergence to the steady state. 
 
 \begin{cor}[Gaussian data: three scenarii]\label{cor:supersonic-gauss} Let $b_0>0$ and $a_0>0$ be given. Define
 \begin{equation}
 \label{psi-infini}
 \psi_\infty:=\ln b_0 -\frac{a_0}{2}\frac{\ln \lambda - \ln a_0}{\lambda -a_0},
 \end{equation}
 with the natural continuation $\psi _\infty= \ln b_0-\frac 12$ if $a_0=\lambda$.
 Denote by $u(t,x)$ the Gaussian solution of Proposition \ref{prop:gauss}.
 \begin{itemize}
\item [(i)] If $\psi _\infty <0$, then there is superexponential
  decay in the sense that 
$$
\Vert u(t,\cdot)\Vert _{L^{\infty}}=\max _{x\in \R} u(t,x) \Eq t\infty
e^{\frac 12}e^{\psi_\infty e^{2\lambda t}}, \quad \Vert
u(t,\cdot)\Vert _{L^1}\Eq t \infty \sqrt{\frac{2\pi}{\lambda}}e^{\frac 1 2} e^{\psi_\infty e^{2\lambda t}}.
$$

\item [(ii)]  If $\psi _\infty =0$, then there is convergence to the steady state of Proposition~\ref{prop:steady}  in the sense that
$$
\Vert u(t,\cdot)-\varphi \Vert _{L^{\infty}} +\Vert u(t,\cdot)-\varphi \Vert _{L^{1}} \Tend t \infty 0.
$$

\item [(iii)]  If $\psi _\infty >0$, then there is superexponential growth in the sense that, for all $R>0$,
$$
\min _{\vert x\vert \leq R} u(t,x) \Eq t\infty  e^{\frac
  12}e^{\psi_\infty e^{2\lambda t}}e^{-\frac{\lambda}{2}R^{2}}, \quad
\Vert u(t,\cdot)\Vert _{L^1}\Eq t \infty  \sqrt{\frac{2\pi}{\lambda}} e^{\frac 12}e^{\psi_\infty e^{2\lambda t}}.
$$
 \end{itemize}
  \end{cor}
  
  \begin{proof}
  One just has to use the asymptotic expansion
  $\psi(t)=\psi_\infty +\frac 1 2 e^{-2\lambda t}+\mathcal O(e^{-4\lambda t})$ as $t\to \infty$, and perform straightforward estimates.
\end{proof}

In view of the comparison principle (Theorem~\ref{theo:comparaison})
and of Corollary~\ref{cor:supersonic-gauss}, we infer:

\begin{cor}[Initial data comparable to a Gaussian]\label{cor:comp-gauss}
  Let $u_0 \in {\rm BPC}^2(\R)$, $u_0\ge 0$, and $\ep\in (0,1)$.  Let
  $a_0,b_0>0$, and denote again
\begin{equation*}
 \psi_\infty=\ln b_0 -\frac{a_0}{2}\frac{\ln \lambda - \ln a_0}{\lambda -a_0},
 \end{equation*}
with the natural continuation $\psi _\infty= \ln b_0-\frac 12$ if $a_0=\lambda$.
  \begin{itemize}
  \item If $\psi_\infty<0$ and $u_0(x)\le b_0e^{-a_0x^2/2}$, then $u$
      decays at least like a double exponential in time,
$$
\Vert u(t,\cdot)\Vert _{L^{\infty}}\le 2 e^{\frac
  12}e^{\psi_\infty e^{2\lambda t}}, \quad \Vert u(t,\cdot)\Vert
_{L^1}\le \sqrt{\frac{4\pi}{\lambda}}e^{\frac 1 2} e^{\psi_\infty
  e^{2\lambda t}}, \quad \text{ as } t\to\infty. 
$$

\item If $\psi_\infty>0$ and $u_0(x)\ge b_0e^{-a_0x^2/2}$, then $u$
   grows locally at least like a double exponential in time: for all $R>0$,
$$
\min _{\vert x\vert \leq R} u(t,x) \ge \frac{1}{2} e^{\frac
  12}e^{\psi_\infty e^{2\lambda t}}e^{-\frac{\lambda}{2}R^{2}}, \quad
\Vert u(t,\cdot)\Vert _{L^1}\ge \sqrt{\frac{\pi}{\lambda}} e^{\frac
  12}e^{\psi_\infty e^{2\lambda t}}, \quad \text{ as } t\to\infty. 
$$
  \end{itemize}
\end{cor}

\begin{rem}\label{rem:dichotomy}
  Observe that, if
$\psi_\infty=\ln b_0 -\frac{a_0}{2}\frac{\ln \lambda - \ln
  a_0}{\lambda -a_0}=0$, then initial data $(1-\ep)b_0e^{-a_0x^2/2}$
($0<\ep<1$), and $(1+\ep)b_0e^{-a_0x^2/2}$ ($\ep>0)$, fall into the regime
$\psi_\infty<0$ (decay), and $\psi_\infty>0$ (growth),
respectively. Typical examples are $(1-\ep)\varphi(x)$,
$(1+\ep)\varphi(x)$, where $\varphi(x)=e^{\frac12}e^{-\lambda x^2/2}$
is the steady state from Proposition~\ref{prop:steady}. 
\end{rem}

\section{Large time behavior in the case of more general data}\label{s:data}

We have seen that the comparison with the constant initial datum equal
to one leads to a strong dichotomy (Corollary~\ref{cor:comp-1}). The 
same is true by comparison with an initial Gaussian, leading to a
larger variety of initial data (Corollary~\ref{cor:comp-gauss} and
Remark~\ref{rem:dichotomy}). 
Now, we enquire on initial data that can be compared neither to 1 nor
to a Gaussian. In this direction, we can prove superexponential decay
for small initial data. To do so, we need the following standard
estimate, which stems from Young inequality applied to the formula
\begin{equation*}
  e^{t\d_{xx}}v_0(x) = \frac{1}{\sqrt{4\pi t}}\int_{\R}e^{-(x-y)^2/(4t)}v_0(y)dy.
\end{equation*}

\begin{lem}
\label{lem:heat} For any initial data $v_0\in L^1(\R )\cap
L^\infty(\R)$, the solution of the Cauchy problem $\partial _t
v=\partial _{xx}v$, $v_{\mid t=0}=v_0$, satisfies
$$
\Vert v(t,\cdot)\Vert _{L^\infty}\leq V(t):=\min\left(\Vert v_0\Vert _{L^\infty},\frac{\Vert v_0\Vert _{L^1}}{\sqrt{4\pi t}}\right), \quad \text{ for any } t\geq 0.
$$
\end{lem}

\begin{theo}[Superexponential decay for small data]
\label{th:donnee-pte} For a nonnegative initial datum $u_0$ in ${\rm
  BPC}^2(\R)$, define $m_\infty:=\Vert u_0\Vert _{L^\infty}$ and
$m_1:=\Vert u_0\Vert _{L^1}$. Assume  
\begin{equation}
\label{donnee-pte}
\begin{aligned}
  \psi_\infty^*:=&\ln m_\infty -\int _{\tau }^{+\infty} \lambda
  e^{-2\lambda s}\ln s \, ds+e^{-2\lambda \tau }\ln \sqrt \tau<0,\\
&\quad \text{ where } \tau:=\left(\frac{m_1}{\sqrt{4\pi}\, m_\infty}\right)^{2}.
\end{aligned}
\end{equation}
Then the solution of \eqref{eq}, starting from $u_0$, is decaying
superexponentially in the sense that 
\begin{equation}
\label{u-donnee-pte}
0\leq u(t,x)\leq \min\left(m_\infty,\frac{m_1}{\sqrt{4\pi t}}\right) e^{\psi(t)e^{2\lambda t}},
\end{equation}
where $\psi(t)\to \psi _\infty^*<0$ as $t\to \infty$.
\end{theo}

\begin{rem} The above criterion provides new initial data leading to
  superexponential decay. For instance, assume that $u_0$ has tails heavier than
  Gaussian (so that domination by a Gaussian cannot be used), that  
\begin{equation}\label{etoile}
1<m_\infty<e^{\int _1^{+\infty} \lambda e^{-2\lambda s}\ln s \, ds},
\end{equation}
(so that domination by the ODE cannot be used), and that
$\tau=1$. Then \eqref{donnee-pte} holds true, hence \eqref{u-donnee-pte}. A typical example could be
  \begin{equation*}
    u_0(x)=  m_\infty e^{-\alpha\(\sqrt{1+x^2}-1\)},
  \end{equation*}
with \eqref{etoile} and $\alpha>0$ adjusted so that 
$\tau=1$. 
\end{rem}

\begin{proof} Following  \cite{Gar-Qui-10} or \cite{Alf-fujita}, we
  look for a supersolution to \eqref{eq} in the form $g(t)v(t,x)$,
  where $g(t)>0$ is to be determined (with $g(0)=1$), and $v(t,x)$ is
  the solution of the heat equation $\partial _t v=\partial _{xx}v$
  with $u_0$ as initial datum. A straightforward computation shows
  that to construct a supersolution,
  it is enough to have  
$$
\frac{g'(t)}{g(t)}-2\lambda \ln g(t)\geq 2\lambda \ln v(t,x).
$$
By  Lemma \ref{lem:heat}, it is therefore enough to have $g(t)=e^{\phi(t)}$ where
$$
\phi'(t)-2\lambda \phi (t)=2\lambda \ln V(t), \quad \phi(0)=0,
$$
that is
$$
\phi(t)=e^{2\lambda t}\int _0 ^{t}2\lambda e^{-2\lambda s}\ln V(s)ds.
$$
Observe that $V(t)=m_\infty$ when $t\leq \tau$ while
$V(t)=\frac{m_1}{\sqrt{4\pi t}}$ when $t\geq \tau$. Cutting the above
integral and performing straightforward computations, we end up  with
$g(t)=e^{\psi(t)e^{2\lambda t}}$, where 
$$
\psi(t):=(\ln m_\infty) (1-e^{-2\lambda \tau})+\ln \frac{m_1}{\sqrt{4\pi}}(e^{-2\lambda \tau}-e^{-2\lambda t})-\int _\tau ^{t}\lambda e^{-2\lambda s}\ln s \, ds,
$$
which tends to $\psi _\infty^*$ as $t\to \infty$. It therefore follows
from the comparison principle (Theorem~\ref{theo:comparaison}) that
$u(t,x)\leq g(t)v(t,x)$, which 
yields \eqref{u-donnee-pte}. 
\end{proof}

In the context of bounded solutions, typically for Lipschitz {\it
  ignition} or {\it bistable} nonlinearities, some threshold results
between extinction and convergence to an equilibrium, say 1, are known
to exist \cite{Zla-06}, \cite{Du-Mat-10}, \cite{Mat-Pol-16}. For
equation \eqref{eq}, we can prove a threshold result between decay and
growth. We first need to construct compactly supported sub-solutions. 

\begin{lem}[High plateaux as sub-solutions]\label{lem:plateaux} Let $L>0$ and $0<\ep<L$ be given. Let $\Theta \in C^\infty([-L,L])\cap {\rm BPC}^2(\R)$ be
  such that 
  \begin{align*}
&\Theta(x)=0\text{ for }|x|\ge L,\\
    & \Theta >0\text{ on }(-L,L),\\
& \Theta(\pm L)=\Theta '(\pm  L)=\Theta ''(\pm L)=0,\\
& \gamma:=\Theta'''(-L)=-\Theta ''' (L)>0,\\
& \Theta \equiv 1 \text{ on }[-L+\ep,L-\ep].
  \end{align*}
Then there is $K_0>1$
  such that, for any $K\geq K_0$, the function $\Theta _K:=K\Theta$
  satisfies 
\begin{equation}\label{sub-plateau}
\Theta _K'' +2\lambda \Theta _K \ln \Theta _K \geq 0,\text{ on }\R.
\end{equation}
Hence, $\Theta _K$ is a sub-solution to \eqref{eq}.  
\end{lem}

\begin{proof}
 By our assumption $\Theta (x)\sim \frac 1 6 \gamma (x+L)^3$,
 $\Theta''(x)\sim \gamma (x+L)$ for $ 0<x+L\ll 1$, where we
 thus have 
 $$
 \Theta _K''(x) +2\lambda \Theta _K(x) \ln \Theta _K(x) \geq K(\Theta''(x)+2\lambda \Theta (x)\ln \Theta (x))\sim K\gamma (x+L).
 $$ 
 As result, there is $\delta >0$ such that \eqref{sub-plateau} holds
 on $(-L,-L+\delta)$ and, by symmetry, on $(L-\delta,L)$.

 Next, denoting $\theta ^*:=\min _{-L+\delta \leq x\leq L -\delta}\Theta (x)>0$, we have, for $x\in [-L+\delta,L-\delta]$, 
 \begin{align*}
 \Theta _K''(x) +2\lambda \Theta _K(x) \ln \Theta _K(x)&= K(\Theta''(x)+2\lambda \Theta (x)\ln \Theta (x)+2\lambda \Theta (x) \ln K)\\
&\geq  K(-\Vert \Theta''+2\lambda \Theta \ln \Theta \Vert _{L^{\infty}}+2\lambda \theta ^{*}\ln K)
 \end{align*}
 which is nonnegative if $K>1$ is large enough.
\end{proof}

\begin{theo}[Threshold phenomena for compactly supported
  data]\label{th:threshold} 

Let $0<\ep<L<L'$ be given. Select $K\geq
  K_0>1$, where $K_0$ is given by Lemma~\ref{lem:plateaux}.
Let $u_0\in {\rm BPC}^2(\R)$ be such that $u_0>0$ on $(-L',L')$ and $u_0\equiv
0$ on $(-\infty,-L']\cup[L',+\infty)$. For $M>0$, we denote by
$u_M(t,x)$ the solution of \eqref{eq} starting from $Mu_0$.  
 \begin{itemize}
\item [(i)] There is $M_{\rm decay}>0$ such that, for any $0<M<M_{\rm
    decay}$, the solution $u_M(t,x)$ is decaying superexponentially in time.

\item [(ii)]  There is $M_{\rm growth}>0$ such that, for any $M>M_{\rm
  growth}$, the solution $u_M(t,x)$ grows locally superexponentially
in time, in the sense that 
\begin{equation}
\label{local-blowup}
u_M(t,x)\geq Ke^{\left(\ln \frac{K+1}{K}\right)e^{2\lambda t}}, \quad
\forall x\in [-L+\ep,L-\ep]. 
\end{equation}
 \end{itemize}
\end{theo}

\begin{proof} The first point is a consequence of
  Corollary~\ref{cor:comp-1}, that is comparison with the ODE, provided we choose $M_{\rm
    decay}=1/\|u_0\|_{L^\infty}$.  We now prove $(ii)$.  Since $\min
  _{-L\leq x\leq L} u_0(x)>0$, there is $M_{\rm growth}>0$ such that, for
  all $M\geq M_{\rm growth}$, $Mu_0\geq (K+1) \Theta$. Next we take $n(t)$
  as the solution of the underlying ODE starting from
  $n_0:=\frac{K+1}{K}$, see \eqref{ode} and \eqref{sol-ode}, that is 
$$
n(t)=e^{\left(\ln \frac{K+1}{K}\right)e^{2\lambda t}}.
$$
 Now we define
$$
w(t,x):=\Theta _K(x)n(t)=K\Theta (x) n(t).
$$
We have $w(0,\cdot)=(K+1)\Theta \leq Mu_0$ and 
\begin{align*}
&\partial_t w(t,x)-\partial _{xx}w(t,x)-2\lambda w(t,x)\ln w(t,x)\\
&= \Theta _K (x)n'(t)-\Theta _K ''(x)n(t)-2\lambda \Theta _K (x) n(t)\ln n(t) -2\lambda \Theta _K (x) n(t)\ln \Theta _K (x)\\
&= n(t)(-\Theta _K ''(x)-2\lambda \Theta _K(x)\ln \Theta _k(x))\\
&\leq  0,
\end{align*}
by Lemma \ref{lem:plateaux}. It therefore follows from the comparison principle that $u_M\geq w$. In particular, since $\Theta \equiv 1$ on $[-L+\ep,L+\ep]$, we get \eqref{local-blowup}. \end{proof}

Proposition~\ref{prop:Omega}, concerned with  a bounded domain, is a straightforward consequence of the
above result. One just has to use a translation in space $c\neq 0$ if necessary so that  $c+[-L',L']\subset
(\alpha,\beta)$, and note that the quantities involved in
Theorem~\ref{th:threshold} control the $L^2$ norm on $\Omega$.

\bibliographystyle{amsplain}

\bibliography{biblio}

\end{document}